\documentclass[12pt]{amsart}

\usepackage{amssymb}
\usepackage{amsmath}
\usepackage{amsthm}
\usepackage{hyperref}
\newcommand{\R}{\mathbb{R}}

\newcommand{\Z}{\mathbb{Z}}

\makeatletter
\def\imod#1{\allowbreak\mkern10mu\left({\operator@font mod}\,\,#1\right)}
\makeatother

\newtheorem*{theorem*}{Theorem}
\newtheorem*{lemma*}{Lemma}
\newtheorem{theorem}{Theorem}[section]
\newtheorem{proposition}[theorem]{Proposition}
\newtheorem{corollary}[theorem]{Corollary}
\newtheorem{lemma}[theorem]{Lemma}
\newtheorem*{conjecture*}{Conjecture}

\newtheorem{Claim*}{Claim}
\theoremstyle{definition}
\newtheorem{definition}[theorem]{Definition}
\newtheorem*{definition*}{Definition}
\newtheorem*{remark*}{Remark}
\numberwithin{equation}{section}

\begin{document}

\title[The Gaussian formula and Fleck's congruence]{A generalization of the Gaussian formula and a q-analog of Fleck's congruence}
\author[A.~Schultz]{Andrew Schultz}
\address{Department of Mathematics, Wellesley College, 106 Central Street, Wellesley, MA 02481}
\email{andrew.c.schultz@gmail.com}
\author[R.~Walker]{Robert Walker}
\address{Department of Mathematics, University of Michigan, 2074 East Hall, 530 Church Street, Ann Arbor, MI 48109}
\email{robmarsw@umich.edu}

\begin{abstract}
The $q$-binomial coefficients are the polynomial cousins of the traditional binomial coefficients, and a number of identities for binomial coefficients can be translated into this polynomial setting.  For instance, the familiar vanishing of the alternating sum across row $n \in \Z_{>0}$ of Pascal's triangle is captured by the so-called Gaussian Formula, which states that $\sum_{m=0}^{n}(-1)^m \binom{n}{m}_q$ is $0$ if $n$ is odd, and is equal to $\prod_{k \mbox{\tiny{ odd}}} (1-q^k)$ if $n$ is even.  In this paper, we find a $q$-binomial congruence which synthesizes this result and Fleck's congruence for binomial coefficients, which asserts that for $n , p \in \mathbb{Z}^+$, with $p$ a prime, 
$$\sum_{m \equiv j \imod{p}}(-1)^m \binom{n}{m} \equiv 0 \imod{p^{\left\lfloor\frac{n-1}{p-1}\right\rfloor}}.$$
\end{abstract}


\date{\today}

\maketitle 

\parskip=10pt plus 2pt minus 2pt

\section{Introduction}

Binomial coefficients are the fundamental objects of enumeration and combinatorics, and they are the subject of myriad identities.  The arithmetic properties of binomial coefficients have also been the subject of extensive study.  Granville gives an excellent account in \cite{G} of some congruence equations involving binomial coefficients modulo powers of primes. The motivation for this paper was to revisit a binomial coefficient congruence which comes from an identity of Fleck (\cite{F}, cf. \cite[p.~274]{D} and \cite{G}) and interpret it through the lens of $q$-binomial coefficients.   

\begin{theorem*}[Fleck's Congruence]
If $n, p \in \Z_{>0}$, $p$ is prime, and $0 \leq j \leq p-1$, then $$\sum_{m \equiv j\imod{p}}(-1)^m\binom{n}{m} \equiv 0 \imod{p^{ \left\lfloor\frac{n-1}{p-1}\right\rfloor}}.$$
\end{theorem*} 

Fleck's congruence has already been the subject of a number of generalizations and analogs (see \cite{Su1,Su2,SD,Wa,We}) and is connected to a variety of mathematical subdisciplines, including algebraic topology \cite{SD}, Iwasawa theory \cite{Wa} and $p$-adic analysis \cite{We}.   When the authors began this project it seemed there was no work which considered $q$-analogs of Fleck's congruence; Pan has since given such a result in \cite{PFleck}.

Recall that $q$-binomial coefficients are polynomials in $\Z[q]$, defined as $$\binom{n}{m}_q = \frac{(1-q^n)(1-q^{n-1})\cdots(1-q^{n-m+1})}{(1-q^m)(1-q^{m-1})\cdots(1-q)},$$ where $m,n \in \Z_{\geq 0}$ and $0 \leq m \leq n$.  There are numerous binomial coefficient identities that have been translated into the language of $q$-binomial coefficients, including a number of recent discoveries (\cite{A,C,P,Sa}). For some classical examples, recall that Lucas' Theorem tells us that if $n,m \in \Z_{\geq 0}$ satisfy $n=pj+a$ and $m = pk+b$, with $0 \leq a,b \leq p-1$ and $p$ a prime, then $$\binom{n}{m} \equiv \binom{j}{k} \binom{a}{b}\imod{p}.$$  The corresponding result in the language of $q$-binomial coefficients is (\cite[p.~131]{Sa})$$\binom{n}{m}_q \equiv \binom{j}{k} \binom{a}{b}_q \imod{\Phi_p (q)}.$$  (See also \cite{HS} for another $q$-generalization of this result.)

The Chu-Vandermonde identity, \begin{equation}\label{eq:chu.vandermonde}\binom{n}{m} = \sum_{j=0}^m \binom{n-t}{m-j}\binom{t}{j},\end{equation} is an essential tool, and it has a $q$-generalization given by $$\binom{n}{m}_q = \sum_{j=0}^m q^{j(n-t-m+j)}\binom{n-t}{m-j}_q \binom{t}{j}_q.$$ A variant of this identity (Lemma \ref{le:critical.identity}) will be of tremendous help as we explore the main theorems of this paper.

As another example, if $n \in \Z_{>0}$, the identity $\sum_{m=0}^n (-1)^m \binom{n}{m} = 0$ has a $q$-analog given by the so-called Gaussian Formula (see, e.g., \cite{AE}): for all $n \in \Z_{>0}$,
$$\sum_{m=0}^n (-1)^m \binom{n}{m}_q = \left\{\begin{array}{ll}\displaystyle \mathop{\prod_{k \mbox{\tiny{ odd}}}}_{1 \le k\leq n} (1-q^k), & \mbox{ if } n \mbox{ is even}\\0, &\mbox{ if } n \mbox{ is odd}.\end{array}\right.$$

Finally, it follows from an identity due to Euler (\cite[pp.~90-91]{LW}),
\begin{equation}\label{eq:Euler Identity}
\sum_{m=0}^n (-1)^m \binom{n}{m} m^l= \left\{\begin{array}{ll}(-1)^n n!, & \mbox{ if } l = n \\0, &\mbox{ if $0 \le l < n$},\end{array}\right.
\end{equation}
that for any $P = P(x) \in \Z[x]$ and $n > \deg(P)$, one has \begin{equation}\label{eq:alt.sum.with.poly}\sum_{m=0}^n(-1)^m P(m)\binom{n}{m} = 0.\end{equation}

The main result of this paper is the following theorem, which can be interpreted as a $q$-synthesis of the Gaussian Formula, Fleck's congruence, and equation (\ref{eq:alt.sum.with.poly}) (though, as we will see, it isn't a full generalization of Fleck's congruence).  We use the notation $\Phi_k (q) \in \Z[q]$ for the $k$-th cyclotomic polynomial and $\zeta_{2c}$ for a fixed primitive $2c$-th root of unity.

\begin{theorem}\label{th:main.theorem}
Suppose that $c,k \in \Z_{> 0}$ with $k$ odd, and that $d, l, z \in \Z_{\geq 0}$. Furthermore, suppose $P \in \Z[\zeta_{2c}][x]$, and that $n$ is an integer satisfying $n \geq (\deg(P)+2(l+d)+1)kc$. Then $$\sum_{m=0}^n \zeta_{2c}^m P(m) (q^z)^m \frac{d^l}{dq^l}\left[\binom{n}{m}_q\right] \equiv 0 \imod{\Phi_{kc}(q)^{d+1}}.$$
\end{theorem}

To recover equation (\ref{eq:alt.sum.with.poly}) from the previous equation, simply set $c=k=P=q=1$ and $d=z=l=0$. As stated, Theorem \ref{th:main.theorem} can't be called a generalization of the Gaussian Formula because it only gives a divisibility result, not an explicit formula.  We will see in the remark following Lemma \ref{le:recursion}, however, that the Gaussian Formula is a consequence of the proof of Theorem \ref{th:main.theorem}; it is in this sense that we say that Theorem \ref{th:main.theorem} is a generalization of the Gaussian Formula.

We will also interpret Theorem \ref{th:main.theorem} as a $q$-analog of Fleck's congruence and consider how close it comes to acting as a generalization in section \ref{sec:fleck}.  For the benefit of the curious reader, we state the $q$-analog of Fleck's congruence below.  We've chosen to restrict to the case where $P(x)$ has integer coefficients since the rest of the sum is integral (though the result holds for $P(x) \in \Z[\zeta_{2c}][x]$ as well).

\begin{theorem}\label{th:fleck.at.k}
Suppose that $c, k \in \Z_{>0}$ with $k$ odd, and that $d, l, z \in \Z_{\geq 0}$. Furthermore, suppose $P(x) \in \Z[x]$, and let $0 \leq j <c$ be given. If $n \geq (\deg(P)+2(l+d)+1)kc$, then 
$$\sum_{m \equiv j \imod{c}}(-1)^{\frac{m-j}{c}} P(m)(q^z)^m \frac{d^l}{dq^l}\left[\binom{n}{m}_q\right] \equiv 0 \imod{\Phi_{kc}(q)^{d+1}}.$$
\end{theorem}

\begin{remark*}
If $m \equiv j \imod{c}$, then $\frac{m-j}{c} \in \Z$. Therefore, since $\zeta_{2c}^c = -1$, 
\[
\begin{split}
\sum_{m=0}^n \zeta_{2c}^m P(m) (q^z)^m \frac{d^l}{dq^l}\left[\binom{n}{m}_q\right] & = \sum_{0 \le j < c} \sum_{m \equiv j \imod{c}} \zeta_{2c}^{m} P(m) (q^z)^m \frac{d^l}{dq^l}\left[\binom{n}{m}_q\right]\\
& = \sum_{0 \le j < c} \sum_{m \equiv j \imod{c}} \zeta_{2c}^{j+ c \left(\frac{m-j}{c}\right)} P(m) (q^z)^m \frac{d^l}{dq^l}\left[\binom{n}{m}_q\right]\\
& = \sum_{0 \le j < c} \zeta_{2c}^j \cdot \sum_{m \equiv j \imod{c}} (-1)^{\frac{m-j}{c}} P(m) (q^z)^m \frac{d^l}{dq^l}\left[\binom{n}{m}_q\right].\\
\end{split}
\]
Therefore, Theorem \ref{th:fleck.at.k} implies Theorem \ref{th:main.theorem}.  We will see that Theorem \ref{th:main.theorem} implies Theorem \ref{th:fleck.at.k} in section \ref{sec:fleck}, and hence these theorems are actually equivalent.\end{remark*}

To avoid unnecessary repetition, in the rest of the paper we use $c$ to denote a positive integer and $k$ an odd natural number.  In the same way, variables $l, d, z$, and $n$ are reserved for natural numbers, with further restrictions stated as necessary.

This paper is organized as follows.  In the next section, we introduce the family $Q^{(l)}_{c}(P(x),z,n)$ of ``alternating sums'' of $q$-binomial coefficients or their derivatives, and we give some basic relations satisfied by these polynomials.  We then devote sections \ref{sec:no.derivative} and \ref{sec:yes.derivative} to proving Theorem \ref{th:main.theorem}, first in the case $l=0$, and then for higher derivatives. In section \ref{sec:fleck}, we explore the connection between Theorem \ref{th:main.theorem} and Fleck's congruence.  Section \ref{sec:wrapup} is spent in discussing some avenues of future exploration.

\section{Some preliminary results}\label{sec:strategy}

As we examine various congruences involving $q$-binomial coefficients and cyclotomic polynomials, we will make frequent use of a few of the basic properties of these polynomials.  We collect these properties in the following lemma, stated without proof (the results stated under this lemma are both well-known and easy to verify).  Moreover, we will not reference this lemma when we use it in proofs later in this paper; the reader will undoubtedly understand the appropriate reference in these cases.

\begin{lemma}\label{le:uncited.lemma}
For $n,m \in \Z_{\geq 0}$, one has
\begin{enumerate}\addtolength{\itemsep}{10pt}
\item $\displaystyle \binom{n+1}{m}_q = q^m\binom{n}{m}_q + \binom{n}{m-1}_q$; 
\item $\displaystyle (1-q^m)\binom{n}{m}_q = (1-q^n)\binom{n-1}{m-1}_q$; and
\item $\Phi_{n}(q)\mid 1-q^m$ if and only if $n \mid m$.
\end{enumerate}
\end{lemma}

The following result will also prove quite handy.  It is the $q$-generalization of the familiar result concerning the divisibility of $\binom{p^\alpha}{m}$ by $p$ when $p$ is prime, $0<m<p^\alpha$ and $\alpha \in \Z_{>0}$. On the one hand, as with a number of $q$-generalizations of binomial coefficient results, the primality requirement isn't necessary for the next lemma.  On the other hand, without this condition the identity doesn't specialize to an interesting identity for binomial coefficients.

\begin{lemma}\label{le:free.divisibility}
If $0<m<n$, then $\Phi_n(q) \mid \binom{n}{m}_q$.
\end{lemma}

\begin{proof}
In the expression for $\binom{n}{m}_q$ there is a factor of $1-q^n$ in the numerator, and hence $\Phi_n(q)$ divides the numerator.  But in the denominator the terms $1-q^j$ do not include any $j$ for which $n \mid j$.  Hence $\Phi_n(q)$ does not divide the denominator.
\end{proof}

Since we've already been introduced to the basic sums of interest, we will give them a name for convenience.

\begin{definition}
For $n,z,c,l \in \Z_{\geq 0}$ and $P(x) \in \Z[\zeta_{2c}][x]$, let
$$Q^{(l)}_c(P(x),z,n) = \sum_{m=0}^n \zeta_{2c}^mP(m) (q^z)^m\frac{d^l}{dq^l}\left[\binom{n}{m}_q\right].$$  Note that $Q^{(l)}_c(P(x),z,n) \in \Z[\zeta_{2c}][q]$.
\end{definition}

We conclude with two results that allow us to express relationships between various polynomials in this family. First,

\begin{lemma}\label{le:n.in.terms.of.n.minus.one.and.two}
For $n \geq 2$, we have
\begin{equation*}\begin{split}Q^{(0)}_c(P(x),0,n) &= Q^{(0)}_c\left(P(x),0,n-1\right)+\zeta_{2c}Q^{(0)}_c\left(P(x+1),0,n-1\right))\\&\quad-\zeta_{2c}(1-q^{n-1})Q^{(0)}_c\left(P(x+1),0,n-2\right).\end{split}\end{equation*}
\end{lemma}

\begin{proof}
This boils down to a calculation that employs a few standard $q$-binomial identities.
\begin{equation*}\begin{split}
Q^{(0)}_c(P(x),0,n) 
&=\sum_{m=0}^n \zeta_{2c}^m P(m)\binom{n}{m}_q \\
&=\sum_{m=0}^n \zeta_{2c}^m P(m)\left(q^m\binom{n-1}{m}_q+\binom{n-1}{m-1}\right) \\
&=\sum_{m=0}^n \zeta_{2c}^m P(m)\binom{n-1}{m}_q-\sum_{m=0}^n \zeta_{2c}^m P(m)(1-q^m)\binom{n-1}{m}_q\\ &\quad+\sum_{m=0}^n \zeta_{2c}^mP(m)\binom{n-1}{m-1}_q \\
&=\sum_{m=0}^n \zeta_{2c}^m P(m)\binom{n-1}{m}_q-(1-q^{n-1})\sum_{m=0}^n \zeta_{2c}^m P(m)\binom{n-2}{m-1}_q\\ &\quad+\sum_{m=0}^n \zeta_{2c}^mP(m)\binom{n-1}{m-1}_q. \\
\end{split}\end{equation*}
The result now follows from shifting indices in the last two sums.
\end{proof}

\begin{remark*}
The previous lemma captures a standard argument for proving the Gaussian Formula.  To see this, one simply needs to observe that when $c=P=1$, the result gives
$$\sum_{m=0}^n (-1)^m \binom{n}{m}_q = (1-q^{n-1}) \sum_{m=0}^{n-2}(-1)^m \binom{n-2}{m}_q.$$  By computing the initial conditions $Q^{(0)}_1(1,0,1) = 0$ and $Q^{(0)}_1(1,0,2) = 1-q$, the result follows by an easy induction.
\end{remark*}

Second, we prove a critical identity that we use in our analysis of Theorem \ref{th:main.theorem}.  This is simply another $q$-analog of Equation (\ref{eq:chu.vandermonde}).

\begin{lemma}[Chu-Vandermonde Analog]\label{le:critical.identity}
$$\sum_{m=0}^n \zeta_{2c}^m P(m) \binom{n}{m}_q = \sum_{m=0}^{n-t} \sum_{j=0}^t \zeta_{2c}^m \zeta_{2c}^{t-j}P(m+t-j)q^{jm}\binom{t}{j}_q \binom{n-t}{m}_q.$$
\end{lemma}

\begin{proof}
We prove the result by induction on $t$.  The base case $t=0$ is trivial, so assume the result holds for $t$, and we show it is true for $t+1$.  Based on our assumption, we have
\begin{equation*}\begin{split}
\sum_{m=0}^n \zeta_{2c}^m &P(m) \binom{n}{m}_q = \sum_{m=0}^{n-t} \sum_{j=0}^t \zeta_{2c}^m \zeta_{2c}^{t-j}P(m+t-j)q^{jm}\binom{t}{j}_q \left(q^m\binom{n-t-1}{m}_q+\binom{n-t-1}{m-1}_q\right) \\
&= \sum_{m=0}^{n-t-1} \sum_{j=0}^t \zeta_{2c}^m \zeta_{2c}^{t-j}P(m+t-j)q^{(j+1)m}\binom{t}{j}_q \binom{n-t-1}{m}_q 
\\&\quad + \sum_{m=0}^{n-t-1} \sum_{j=0}^t \zeta_{2c}^{m+1} \zeta_{2c}^{t-j}P(m+1+t-j)q^{j(m+1)}\binom{t}{j}_q\binom{n-t-1}{m}_q \\
&= \sum_{m=0}^{n-t-1} \sum_{j=0}^t \zeta_{2c}^m \zeta_{2c}^{t-j}P(m+t-j)q^{(j+1)m} \binom{t}{j}_q \binom{n-t-1}{m}_q 
\\&\quad + \sum_{m=0}^{n-t-1} \sum_{j=0}^t \zeta_{2c}^{m+1} \zeta_{2c}^{t-j}P(m+1+t-j)q^{jm}\left(\binom{t+1}{j}_q-\binom{t}{j-1}_q\right)\binom{n-t-1}{m}_q.
\end{split}\end{equation*}
By shifting indices once again, we're left with
\begin{equation*}\begin{split}
\sum_{m=0}^n \zeta_{2c}^m &P(m) \binom{n}{m}_q = \sum_{m=0}^{n-t-1} \sum_{j=0}^t \zeta_{2c}^m \zeta_{2c}^{t-j}P(m+t-j)q^{(j+1)m} \binom{t}{j}_q \binom{n-t-1}{m}_q 
\\&\quad + \sum_{m=0}^{n-t-1} \sum_{j=0}^t \zeta_{2c}^{m+1} \zeta_{2c}^{t-j}P(m+1+t-j)q^{jm}\binom{t+1}{j}_q\binom{n-t-1}{m}_q
\\&\quad -\sum_{m=0}^{n-t-1} \sum_{j=0}^{t-1} \zeta_{2c}^{m+1} \zeta_{2c}^{t-j-1}P(m+t-j)q^{(j+1)m}\binom{t}{j}_q\binom{n-t-1}{m}_q \\
&=\sum_{m=0}^{n-t-1} \zeta_{2c}^m P(m)q^{(t+1)m} \binom{n-t-1}{m}_q \\&\quad + \sum_{m=0}^{n-t-1} \sum_{j=0}^t \zeta_{2c}^{m+1} \zeta_{2c}^{t-j}P(m+1+t-j)q^{jm}\binom{t+1}{j}_q\binom{n-t-1}{m}_q \\
&= \sum_{m=0}^{n-t-1}\sum_{j=0}^{t+1} \zeta_{2c}^m \zeta_{2c}^{t+1-j}P(m+t+1-j)q^{jm}\binom{t+1}{j}_q \binom{n-t-1}{m}_q.
\end{split}\end{equation*}
This proves the lemma.
\end{proof}

\section{The case $l=0$}\label{sec:no.derivative}

In this section, we prove Theorem \ref{th:main.theorem} when $l=0$.  The basic strategy is to proceed by induction on the various parameters of interest: $n,\deg(P), d$, and $z$. Throughout the section, we use $c$ to denote a positive integer and $k$ an odd natural number.  In the same way, variables $l, d, z$, and $n$ are reserved for natural numbers.  The term $\zeta_{2c}$ stands for a fixed, primitive $2c$-th root of unity.

\subsection{First order vanishing}
\begin{lemma}\label{le:d.zero.and.degree.P.zero}
If $a \in \Z[\zeta_{2c}]$ and $n \geq kc$, then $Q^{(0)}_c(a,0,n) \equiv 0 \imod{\Phi_{kc}(q)}$.
\end{lemma}
\begin{proof}
We induct on $n$, starting with $n=kc$.  Since
$$Q^{(0)}_c(a,0,kc) = a \sum_{m=0}^{kc} (\zeta_{2c})^m \binom{kc}{m}_q,$$
Lemma \ref{le:free.divisibility} tells us that every term with $0<m<kc$ is divisible by $\Phi_{kc}(q)$.  The summands corresponding to $m=0$ and $m=kc$ then combine to give
$$Q^{(0)}_c(a,0,kc) \equiv a(\zeta_{2c})^0\binom{kc}{0}_q + a(\zeta_{2c})^{kc}\binom{kc}{kc}_q \equiv a-a \equiv 0 \imod{\Phi_{kc}(q)}.$$  Thus the result follows in this case.

When $n = kc+1$, Lemma \ref{le:n.in.terms.of.n.minus.one.and.two} tells us that
$$Q^{(0)}_{c}(a,0,kc+1) = (1+\zeta_{2c})Q^{(0)}_c\left(a,0,kc\right)-\zeta_{2c}(1-q^{kc})Q^{(0)}_c(a,0,kc-1).$$ Since $Q^{(0)}_c(a,0,kc) \equiv 0 \imod{\Phi_{kc}(q)}$ by the previous case, the first term vanishes modulo $\Phi_{kc}(q)$.  The second summand vanishes modulo $\Phi_{kc}(q)$ because $\Phi_{kc}(q) \mid (1-q^{kc})$. 

Now when $n \geq kc+2$, we again apply Lemma \ref{le:n.in.terms.of.n.minus.one.and.two} and notice that both terms vanish modulo $\Phi_{kc}(q)$ by induction.
\end{proof}

\begin{lemma}\label{le:d.zero.and.z.zero}
If $P \in \Z[\zeta_{2c}][x]$ and $n \geq (\deg(P)+1)kc$, then $Q^{(0)}_c(P(x),0,n) \equiv 0 \imod{\Phi_{kc}(q)}$.
\end{lemma}

\begin{proof}
We prove this by induction on $\deg(P)$, with the case $\deg(P) = 0$ handled in Lemma \ref{le:d.zero.and.degree.P.zero}.  So suppose that we know $Q^{(0)}_c(P(x),0,n) \equiv 0 \imod{\Phi_{kc}(q)}$ when $\deg(P)<b$ and $n \geq (\deg(P)+1)kc$, and we study $Q^{(0)}_c(P(x),0,n)$ when $\deg(P)=b$ and $n \geq (b+1)kc$.

We prove the result by induction on $n$, starting with $n=(b+1)kc$. In this case, using the Chu-Vandermonde Analog (Lemma \ref{le:critical.identity}) with $t=kc$ gives
\begin{equation*}\begin{split}
Q^{(0)}_c(P(x),0,n) &= \sum_{m=0}^{n-kc} \sum_{j=0}^{kc} \zeta_{2c}^m\zeta_{2c}^{kc-j}P(m+kc-j)q^{jm}\binom{kc}{j}_q \binom{n-kc}{m}_q \\
=&\sum_{j=0}^{kc} \zeta_{2c}^{kc-j}\binom{kc}{j}_q  \sum_{m=0}^{n-kc} \zeta_{2c}^mP(m+kc-j)q^{jm}\binom{n-kc}{m}_q.
\end{split}\end{equation*}
Each of the summands with $0<j<kc$ is divisible by $\Phi_{kc}(q)$ by Lemma \ref{le:free.divisibility}, and hence we have the following congruence relation modulo $\Phi_{kc}(q)$:
\begin{equation*}\begin{split}
Q^{(0)}_c(P(x),0,n) &\equiv \sum_{m=0}^{n-kc} \zeta_{2c}^mP(m)q^{kcm}\binom{n-kc}{m}_q - \sum_{m=0}^{n-kc}\zeta_{2c}^mP(m+kc)\binom{n-kc}{m}_q \\
&\equiv \sum_{m=0}^{n-kc} \zeta_{2c}^mP(m)(q^{kcm}-1)\binom{n-kc}{m}_q + \sum_{m=0}^{n-kc}\zeta_{2c}^m(P(m)-P(m+kc))\binom{n-kc}{m}_q. \\
\end{split}\end{equation*}
Now the latter sum is $Q^{(0)}_c(P(x)-P(x+kc),0,n-kc)$, and since $\deg(P(x)-P(x+kc))\leq b-1$ with $n-kc \geq ((b-1)+1)kc$, induction on $\deg(P)$ tells us that this summand is congruent to $0$ modulo $\Phi_{kc}(q)$.  As for the former sum, note that $q^{kcm}-1 = (q^m-1)\sum_{i=0}^{kc-1} q^{im}$.  Hence
\begin{equation*}\begin{split}
Q^{(0)}_c(P(x),0,n) &\equiv \sum_{m=0}^{n-kc} \zeta_{2c}^mP(m)(q^{kcm}-1)\binom{n-kc}{m}_q \\
&\equiv (q^{n-kc}-1)\sum_{i=0}^{kc-1}\sum_{m=0}^{n-kc} \zeta_{2c}^mP(m)q^{im}\binom{n-kc-1}{m-1}_q \imod{\Phi_{kc}(q)}.
\end{split}\end{equation*}  Since $kc \mid n$ we have $\Phi_{kc}(q)\mid (q^{n-kc}-1)$, and the result follows. This resolves the case $n=(b+1)kc$.

Now when $n = (b+1)kc+1$, Lemma \ref{le:n.in.terms.of.n.minus.one.and.two} tells us that 
$$Q^{(0)}_c(P(x),0,n) = Q^{(0)}_c\left(P(x)+\zeta_{2c}P(x+1),0,n-1\right)-\zeta_{2c}(1-q^{n-1})Q^{(0)}_c(P(x+1),0,n-2).$$
Note that since $n=(b+1)kc+1$, the first summand is congruent to $0$ modulo $\Phi_{kc}(q)$ by the previous base case on $n$.  As for the second summand, the factor $1-q^{n-1}$ contains a factor of $\Phi_{kc}(q)$. Hence this second summand is also congruent to $0$ modulo $\Phi_{kc}(q)$.

Finally, when $n \geq (b+1)kc+2$, Lemma \ref{le:n.in.terms.of.n.minus.one.and.two} allows us to use induction on $n$ to verify the desired congruence.
\end{proof}

\begin{lemma}\label{le:d.zero.in.general}
If $n \geq (\deg(P)+1)kc$, then $Q^{(0)}_c(P(x),z,n) \equiv 0 \imod{\Phi_{kc}(q)}$.
\end{lemma}

\begin{proof}
We prove this result by induction on $z \in \Z_{\geq 0}$, with Lemma \ref{le:d.zero.and.z.zero} as the base case.  So assume that $\Phi_{kc}(q) \mid Q^{(0)}_c(P(x),z,n)$ for every $n \geq (\deg(P)+1)kc$, and we consider $Q^{(0)}_c(P(x),z+1,n)$.  We have
\begin{equation}\label{eq:z.plus.one.in.terms.of.z}\begin{split}
Q^{(0)}_c(P(x),z+1,n) &= \sum_{m=0}^n \zeta_{2c}^m P(m)q^{zm}\left(\binom{n+1}{m}_q-\binom{n}{m-1}_q\right) \\
&= \sum_{m=0}^n \zeta_{2c}^m P(m)q^{zm}\binom{n+1}{m}_q-\zeta_{2c}q^z\sum_{m=0}^{n-1} \zeta_{2c}^m P(m+1)q^{zm}\binom{n}{m}_q \\
&= Q^{(0)}_c(P(x),z,n+1)-\zeta_{2c}^{n+1}P(n+1)q^{z(n+1)} \\& \quad -\zeta_{2c}q^zQ^{(0)}_c(P(x+1),z,n) + \zeta_{2c}q^z \zeta_{2c}^n P(n+1)q^{nz} \\
&= Q^{(0)}_c(P(x),z,n+1)-\zeta_{2c}q^zQ^{(0)}_c(P(x+1),z,n).
\end{split}\end{equation}
Since these two terms are each divisible by $\Phi_{kc}(q)$ by induction, the result follows.
\end{proof}

\subsection{Higher order vanishing}

Now that we've shown that the polynomials of interest are divisible by the given cyclotomic factors, we account for multiplicities.  Again, the key result will be Lemma \ref{le:critical.identity}.

\begin{lemma}\label{le:base.case.for.higher.vanishing}
If $n \geq (2d+1)kc$ and $a \in \Z[\zeta_{2c}]$, then $Q^{(0)}_c(a,z,n) \equiv 0 \imod{\Phi_{kc}(q)^{d+1}}$.
\end{lemma}

\begin{proof}
We verify the result by induction on $d \in \Z_{\geq 0}$.  The base case $d=0$ is handled in Lemma \ref{le:d.zero.in.general}, so suppose we know that $Q^{(0)}_c(a,z,n) \equiv 0 \imod{\Phi_{kc}(q)^d}$ for $n \geq (2(d-1)+1)kc$, and we consider $Q^{(0)}_c(a,z,n)$ for $n \geq (2d+1)kc$.

We begin with $z=0$.  We prove the result by induction on $n$, starting with $n=(2d+1)kc$. In this case, by the Chu-Vandermonde Analog (Lemma \ref{le:critical.identity}) with $t=kc$, we find
$$Q^{(0)}_c(a,0,n) = \sum_{j=0}^{kc}\zeta^{kc-j}_{2c} \binom{kc}{j}_q Q^{(0)}_c(a,j,n-kc).$$  For $0<j<kc$ the term $\binom{kc}{j}_q$ contains a factor of $\Phi_{kc}(q)$, and by induction on $d$ we know that $Q^{(0)}_c(a,j,n-kc)$ is divisible by $\Phi_{kc}(q)^d$.  Hence we have
\begin{equation*}\begin{split}Q^{(0)}_c(a,0,n) &\equiv a \sum_{m=0}^{n-kc}\zeta_{2c}^m(q^{kcm}-1)\binom{n-kc}{m}_q\\
&\equiv a(q^{n-kc}-1)\sum_{i=0}^{kc-1} \sum_{m=0}^{n-kc} \zeta_{2c}^m q^{im} \binom{n-kc-1}{m-1}_q\\
&\equiv (q^{n-kc}-1)\sum_{i=0}^{kc-1}\zeta_{2c}q^i Q^{(0)}_c(a,i,n-kc-1) \imod{\Phi_{kc}(q)^{d+1}}.
\end{split}\end{equation*}
Since $kc \mid n$, the term $q^{n-kc}-1$ contains one factor of $\Phi_{kc}(q)$, and each of the terms $Q^{(0)}_c(a,i,n-kc-1)$ contains a factor of $\Phi_{kc}(q)^d$ by induction on $d$.  Hence the result follows when $n=(2d+1)kc$.

Now when $n = (2d+1)kc+1$, Lemma \ref{le:n.in.terms.of.n.minus.one.and.two} tells us that 
$$Q^{(0)}_{c}(a,0,n) = (1+\zeta_{2c})Q^{(0)}_c\left(a,0,n-1\right)-\zeta_{2c}(1-q^{n-1})Q^{(0)}_c(a,0,n-2).$$
Note that since $n=(2d+1)kc+1$, the first summand is congruent to $0$ modulo $\Phi_{kc}(q)^{d+1}$ by the previous base case on $n$.  As for the second summand, the factor $1-q^{n-1}$ is divisible by $\Phi_{kc}(q)$, whereas the term $Q^{(0)}_c(a,0,n-2)$ is divisible by $\Phi_{kc}(q)^{d}$ by induction on $d$.  Hence this second summand is also congruent to $0$ modulo $\Phi_{kc}(q)^{d+1}$.

Finally, when $n \geq (2d+1)kc+2$, Lemma \ref{le:n.in.terms.of.n.minus.one.and.two} allows us to use induction on $n$ to verify the desired congruence.

For $z \in \Z_{\geq 0}$ arbitrary, notice that equation (\ref{eq:z.plus.one.in.terms.of.z}) gives
$$Q^{(0)}_c(a,z+1,n)= Q^{(0)}_c(a,z,n+1)-\zeta_{2c}q^zQ^{(0)}_c(a,z,n),$$
and hence the result follows by induction on $z$.
\end{proof}

\begin{theorem}\label{th:big.theorem.when.z.zero}
If $n \geq (\deg(P)+2d+1)kc$, then $Q^{(0)}_c(P(x),z,n) \equiv 0 \imod{\Phi_{kc}(q)^{d+1}}$.
\end{theorem}
\begin{proof}
We prove this result by induction on $d$, with the case $d=0$ handled in Lemma \ref{le:d.zero.in.general}.  So suppose that we know $\Phi_{kc}(q)^d \mid Q^{(0)}_c(P(x),z,n)$ whenever $n \geq (\deg(P)+2(d-1)+1)kc$, and we consider $Q^{(0)}_c(P(x),z,n)$ when $n \geq (\deg(P)+2d+1)kc$.  

We proceed by induction on $\deg(P)$.  When $P \in \Z[\zeta_{2c}]$ (i.e., $\deg(P) = 0$), the result follows from Lemma \ref{le:base.case.for.higher.vanishing}.  So assume we know that $\Phi_{kc}(q)^{d+1} \mid Q^{(0)}_c(P(x),z,n)$ whenever $\deg(P)<b$ and $n \geq (\deg(P)+2d+1)kc$, and we consider $Q^{(0)}_c(P(x),z,n)$ when $\deg(P)=b$ and $n \geq (\deg(P)+2d+1)kc$.

We will proceed by induction on $z$, and so first consider $z=0$.  We prove the result by induction on $n$, starting with $n=(\deg(P)+2d+1)kc$.  We apply the Chu-Vandermonde Analog (Lemma \ref{le:critical.identity}) when $t=kc$:
\begin{equation*}\begin{split}
\sum_{m=0}^n \zeta_{2c}^m P(m) \binom{n}{m}_q &= \sum_{m=0}^{n-kc} \sum_{j=0}^{kc} \zeta_{2c}^m\zeta_{2c}^{kc-j}P(m+kc-j)q^{jm}\binom{kc}{j}_q \binom{n-kc}{m}_q \\
&= \sum_{j=0}^{kc}\zeta_{2c}^{kc-j}\binom{kc}{j}_q Q_c^{(0)}(P(x+kc-j),j,n-kc).\\
\end{split}\end{equation*}
For each of the terms $0<j<kc$ we know that $\Phi_{kc}(q) \mid \binom{kc}{j}_q$, and furthermore we know that $\Phi_{kc}(q)^d \mid Q^{(0)}_c(P(x+kc-j),j,n-kc)$ since $n-kc \geq (\deg(P)+2(d-1)+1)kc$.  Hence each of these terms vanishes modulo $\Phi_{kp}(q)^{d+1}$, and we are left with
\begin{equation*}\begin{split}
Q^{(0)}_c(P(x),0,n) &\equiv \sum_{m=0}^{n-kc}\zeta_{2c}^mP(m)q^{kcm}\binom{n-kc}{m}_q - \sum_{m=0}^{n-kc}\zeta_{2c}^m P(m+kc)\binom{n-kc}{m}_q \\
&\equiv \sum_{m=0}^{n-kc}\zeta_{2c}^mP(m)(q^{kcm}-1)\binom{n-kc}{m}_q + \sum_{m=0}^{n-kc}\zeta_{2c}^m \left(P(m)- P(m+kc)\right)\binom{n-kc}{m}_q \\
&\equiv (q^{n-kc}-1)\sum_{i=0}^{kc-1}\sum_{m=0}^{n-kc}\zeta_{2c}^mP(m)q^{im}\binom{n-kc-1}{m-1}_q \\&\quad + \sum_{m=0}^{n-kc}\zeta_{2c}^m \left(P(m)- P(m+kc)\right)\binom{n-kc}{m}_q \\
&\equiv \zeta_{2c}(q^{n-kc}-1)\sum_{i=0}^{kc-1}q^iQ^{(0)}_c(P(x+1),i,n-kc-1) \\&\quad + \sum_{m=0}^{n-kc}\zeta_{2c}^m \left(P(m)- P(m+kc)\right)\binom{n-kc}{m}_q \imod{\Phi_{kc}(q)^{d+1}}. \\
\end{split}\end{equation*}
Note that in the latter sum $\deg(P(x)-P(x+kc))<b$ and $n-kc \geq (\deg(P)+2d+1)kc-kc \geq (\deg(P(x)-P(x+kc))+2d+1)kc$, and hence by induction on $\deg(P)$ this term vanishes modulo $\Phi_{kc}(q)^{d+1}$.  For the former sum, we have $n-kc-1 \geq (\deg(P)+2(d-1)+1)kc$, and hence $\Phi_{kc}(q)^d \mid Q^{(0)}_c(P(x+1),i,n-kc-1)$ by induction on $d$.  Since $kc \mid n$ we have $\Phi_{kc}(q) \mid q^{n-kc}-1$, and hence the first summand is also congruent to $0$ modulo $\Phi_{kc}(q)^{d+1}$.  This resolves the case $n=(\deg(P)+2d+1)kc$.

Now when $n = (\deg(P)+2d+1)kc+1$, Lemma \ref{le:n.in.terms.of.n.minus.one.and.two} tells us that 
$$Q^{(0)}_c(P(x),0,n) = Q^{(0)}_c\left(P(x)+\zeta_{2c}P(x+1),0,n-1\right)-\zeta_{2c}(1-q^{n-1})Q^{(0)}_c(P(x+1),0,n-2).$$
Note that since $n=(\deg(P)+2d+1)kc+1$, the first summand is congruent to $0$ modulo $\Phi_{kc}(q)^{d+1}$ by the previous case.  For the second summand, the factor $1-q^{n-1}$ contains a factor of $\Phi_{kc}(q)$, whereas the term $Q^{(0)}_c(P(x+1),0,n-2)$ is divisible by $\Phi_{kc}(q)^{d}$ by induction on $d$.  Hence this second summand is also congruent to $0$ modulo $\Phi_{kc}(q)^{d+1}$.

Finally, when $n \geq (\deg(P)+2d+1)kc+2$, Lemma \ref{le:n.in.terms.of.n.minus.one.and.two} allows us to use induction on $n$ to verify the desired congruence.

With the base case $z=0$ resolved, we go on to consider $z>0$.  Note that from equation (\ref{eq:z.plus.one.in.terms.of.z}) we have 
$$Q^{(0)}_c(P(x),z+1,n)= Q^{(0)}_c(P(x),z,n+1)-\zeta_{2c}q^zQ^{(0)}_c(P(x+1),z,n).$$
By induction on $z$, the result holds.
\end{proof}

\section{Higher derivatives}\label{sec:yes.derivative}
In this section, we again use $c$ to denote a positive integer and $k$ an odd natural number.  In the same way, variables $l, d, z$, and $n$ are reserved for natural numbers, and $\zeta_{2c}$ is a fixed, primitive $2c$-th root of unity. 

To finish the proof of Theorem \ref{th:main.theorem}, we need to account for the cases where $l>0$.  Since the case $l=0$ was handled in the previous section, we use induction on $l$.  Hence we may assume we know that $Q^{(v)}_c(P(x),z,n) \equiv 0 \imod{\Phi_{kc} (q) ^{d+1}}$ whenever $v<l$ and $n \geq (\deg(P)+2(v+d)+1)kc$, and we would like to consider $Q^{(l)}_c(P(x),z,n)$ when $n \geq (\deg(P)+2(l+d)+1)kc$.

We prove the result by induction on $z$.  Notice that $$Q^{(l)}_c(P(x),0,n) = \sum_{m=0}^n \zeta_{2c}^m P(m) \frac{d^l}{dq^l}\left[\binom{n}{m}_q\right] = \frac{d^l}{dq^l}\left[Q^{(0)}_c(P(x),0,n)\right].$$  By Theorem \ref{th:big.theorem.when.z.zero} we know that $\Phi_{kc}(q)^{d+l+1} \mid Q^{(0)}_c(P(x),0,n)$, and hence $Q^{(l)}_c(P(x),0,n)$ is divisible by $\Phi_{kc}(q)^{d+1}$ as desired. So assume that $Q^{(l)}_c(P(x),z,n) \equiv 0 \imod{\Phi_{kc}(q)^{d+1}}$ for all $n \geq (\deg(P)+2(l+d)+1)kc$, and consider $Q^{(l)}_c(P(x),z+1,n)$. 

Before continuing, we notice that the product rule allows us to express
\begin{equation*}\begin{split}q^x\frac{d^l}{dq^l}\left[\binom{n}{x}_q\right] &= \frac{d^l}{dq^l}\left[q^x\binom{n}{x}_q\right] - \sum_{j=1}^{l} \binom{l}{j} \frac{d^{j}}{dq^{j}}\left[q^x\right]\frac{d^{l-j}}{dq^{l-j}}\left[\binom{n}{x}_q\right]\\
&=\frac{d^l}{dq^l}\left[q^x\binom{n}{x}_q\right] - \sum_{j=1}^{l} R_j(x) q^{x-j} \frac{d^{l-j}}{dq^{l-j}}\left[\binom{n}{x}_q\right]
\end{split}\end{equation*} where $R_{j}(x) = \binom{l}{j}\prod_{k=0}^{j-1}(x-k) \in \Z[x]$.  Hence we have
\begin{equation}\label{eq:deriv.relation}\begin{split}
Q^{(l)}_c(P(x),&z+1,n) = \sum_{m=0}^n \zeta_{2c}^m P(m) q^{zm} q^m \frac{d^l}{dq^l}\left[\binom{n}{m}_q\right] \\
&=\sum_{m=0}^n \zeta_{2c}^m P(m) q^{zm} \left(\frac{d^l}{dq^l}\left[q^m\binom{n}{m}_q\right] - \sum_{j=1}^{l} R_j(m)q^{m-j} \frac{d^{l-j}}{dq^{l-j}}\left[\binom{n}{m}_q\right]\right) \\
&=\sum_{m=0}^n \zeta_{2c}^m P(m) q^{zm} \left(\frac{d^l}{dq^l}\left[\binom{n+1}{m}_q - \binom{n}{m-1}_q\right]\right) \\& \quad - \sum_{j=1}^{l} q^{-j} \sum_{m=0}^n \zeta_{2c}^m P(m)R_j(m) q^{(z+1)m} \frac{d^{l-j}}{dq^{l-j}}\left[\binom{n}{m}_q\right]\\
&=Q^{(l)}_c(P(x),z,n+1) - \zeta_{2c}q^z Q^{(l)}_c(P(x+1),z,n) \\&\quad - \sum_{j=1}^{l} q^{-j}Q^{(l-j)}_c(P(x)R_j(x),z+1,n).
\end{split}\end{equation}

We pause briefly to make sense of the terms $q^{-j} Q^{(l-j)}_c(P(x)R_j(x),z+1,n)$, where $1 \leq j \leq l$.  Notice that $R_j(x) = 0$ for $x \in \Z$ with $0 \leq x \leq j-1$, and hence each of the summands in $Q^{(l-j)}_c(P(x)R_j(x),z+1,n)$ is divisible by $q^j$.  Hence $q^j \mid Q^{(l-j)}_c(P(x)R_j(x),z+1,n)$, and so $q^{-j} Q^{(j)}_c(P(x)R_j(x),z+1,n) \in \Z[\zeta_{2c}][q]$.  Moreover, since $\deg(P(x)R_j(x)) = \deg(P)+j$, and because $n \geq (\deg(P(x)R_j(x))+2(l-j+d)+1)kc$, by induction on $l$ we know that $$Q^{(l-j)}_c(P(x)R_j(x),z+1,n) \equiv 0 \imod{\Phi_{kc}(q)^{d+1}}.$$  Because $(q,\Phi_{kc}(q)) = 1$, it also follows that 
\begin{equation*}\begin{split}
Q^{(l)}_c(P(x),z+1,n) & \equiv Q^{(l)}_c(P(x),z,n+1) - \zeta_{2c}q^z Q^{(l)}_c(P(x+1),z,n) \imod{\Phi_{kc}(q)^{d+1}}.
\end{split}\end{equation*}
By induction on $z$, both of these terms are divisible by $\Phi_{kc}(q)^{d+1}$. This proves the theorem.

\section{Fleck-type sums}\label{sec:fleck}  
In this section, we again use $c$ to denote a positive integer and $k$ an odd natural number.  In the same way, variables $l, d, z$, and $n$ are reserved for natural numbers, and $\zeta_{2c}$ stands for a fixed, primitive $2c$-th root of unity. We devote this section to proving Theorem \ref{th:fleck.at.k} and consider how close it comes to providing a generalization of Fleck's congruence. We will see that Theorem \ref{th:fleck.at.k} provides ``half" of a generalization of Fleck's congruence, and that we can recover a result analogous to Fleck's congruence when $p=2$ only by forcing the sum under consideration to be alternating.

\subsection{The $q$-analog of Fleck's congruence}

\begin{proof}[Proof of Theorem \ref{th:fleck.at.k}]
Recall that we're assuming that $n \geq (\deg(P)+2(l+d)+1)kc$, and our goal is to prove that for any $0 \leq j \leq c-1$ we have
$$\sum_{m \equiv j \imod{c}} (-1)^{\frac{m-j}{c}} P(m)(q^z)^m \frac{d^l}{dq^l}\left[\binom{n}{m}_q\right] \equiv 0 \imod{\Phi_{kc}(q)^{d+1}}.$$

For $0 \leq h \leq c-1$, consider
$$G_h = \sum_{m=0}^n (\zeta_{2c}^{1+2h})^m P(m)(q^z)^m \frac{d^l}{dq^l}\left[\binom{n}{m}_q\right].$$
Notice that $G_0 = Q^{(l)}_c(P(x),z,n)$, and so Theorem \ref{th:main.theorem} gives $G_0 \equiv 0 \imod{\Phi_{kc}(q)^{d+1}}$.  Similarly if $(1+2h)r = c$, then we have $G_h = Q^{(l)}_r(P(x),z,n)$, and since $kc = k(1+2h)r$ is an odd multiple of $r$, Theorem \ref{th:main.theorem} again gives $G_h \equiv 0 \imod{\Phi_{kc}(q)^{d+1}}$.  Now suppose we have $(1+2h,c) = s$, so that $\frac{c}{s} =r$, but that $1+2h \nmid c$.  Then the term $\zeta_{2c}^{1+2h}$ in the expression for $G_h$ is still a primitive $2r$-th root of unity.  Since the proof of Theorem \ref{th:main.theorem} doesn't depend on the choice of a primitive $2r$-th root of unity, we can still apply Theorem 1.1.  Now since $s$ is an odd number, observe that $kc = ksr$ is an odd multiple of $r$, and so we conclude that for all $0 \leq h \leq c-1$ we have $G_h \equiv 0 \imod{\Phi_{kc}(q)^{d+1}}$.

Notice, however, that 
\begin{equation*}\begin{split}
\sum_{h=0}^{c-1} \zeta_{c}^{-jh} G_h &= \sum_{m=0}^n \zeta_{2c}^mP(m)(q^z)^m \frac{d^l}{dq^l}\left[\binom{n}{m}_q\right] \sum_{h=0}^{c-1} (\zeta_{c}^{m-j})^h \\
&= c \cdot \zeta_{2c}^j \sum_{m\equiv j \imod{c}} \zeta_{2c}^{m-j} P(m)(q^z)^m\frac{d^l}{dq^l}\left[\binom{n}{m}_q\right] \\
&= c  \cdot \zeta_{2c}^j \sum_{m\equiv j \imod{c}} (-1)^{\frac{m-j}{c}} P(m)(q^z)^m\frac{d^l}{dq^l}\left[\binom{n}{m}_q\right].
\end{split}\end{equation*}
Since $(c,\Phi_{kc}(q))=1$ and $\zeta_{2c}$ is a unit, we have the desired divisibility.
\end{proof}

In the following proposition, $\lfloor x \rceil$ is defined to be the nonnegative integer closest to $x \in \R$, where $\lfloor x \rceil=0$ for all $x < 1/2$ and $\lfloor n + \frac{1}{2} \rceil = n+1$ for $n \in \Z_{\geq 0}$. 

\begin{proposition}\label{prop:fleck.restated}
If $n \in \Z_{\geq 0}$, $0 \leq j < c$, and $P \in \Z[x]$, then
$$\sum_{m \equiv j \imod{c}}(-1)^{\frac{m-j}{c}} P(m)(q^z)^m \frac{d^l}{dq^l}\left[\binom{n}{m}_q\right] \equiv 0 \imod{\mathop{\prod_{k \mbox{\tiny{ odd}}}}\Phi_{kc}(q)^{\varepsilon(kc,l,P(x),n)}},$$ where $\varepsilon(kc,l,P(x),n) = \left\lfloor \frac{n}{2kc}-\frac{\deg(P)}{2}-l\right\rceil$.
\end{proposition}

\begin{proof}
For an odd $k\in \Z_{\geq 0}$ given, if $d_k\in \Z_{\geq 0}$ can be chosen so that $$(\deg(P)+2(l+d_k+1)+1)kc > n \geq (\deg(P)+2(l+d_k)+1)kc,$$  then dividing this by $2kc$ and rearranging gives 
$$d_k+1.5 > \frac{n}{2kc} -\frac{\deg(P)}{2} - l \geq d_k+0.5.$$ Applying the rounding function $\left\lfloor \cdot \right\rceil$ preserves the inequalities (since the term $d_k+1.5$ is certainly rounded up), and so one finds that $d_k+1 = \left\lfloor \frac{n}{2kc}-\frac{\deg(P)}{2}-l\right\rceil = \varepsilon(kc,l,P(x),n)$.  Theorem \ref{th:fleck.at.k} then gives $$\sum_{m \equiv j \imod{c}}(-1)^{\frac{m-j}{c}}P(m)(q^z)^m\frac{d^l}{dq^l}\left[\binom{n}{m}_q\right] \equiv 0 \imod{\Phi_{kc}(q)^{\varepsilon(kc,l,P(x),n)}}.$$  Otherwise, $\frac{n}{2kc}-\frac{\deg(P)}{2}-l < 1/2$, hence $\varepsilon(kc,l,P(x),n) = 0$, and it is still the case that $$\sum_{m \equiv j \imod{c}}(-1)^{\frac{m-j}{c}}P(m)(q^z)^m\frac{d^l}{dq^l}\left[\binom{n}{m}_q\right] \equiv 0 \imod{\Phi_{kc}(q)^{\varepsilon(kc,l,P(x),n)}}.$$ Since $k$ was arbitrary, we can assemble these various congruences together with the Chinese Remainder Theorem because cyclotomic polynomials are pairwise relatively prime.
\end{proof}

\subsection{Recovering a binomial congruence when $c=p^\alpha$}

When we substitute $q=1$ and $l=0$ into Proposition \ref{prop:fleck.restated}, we recover a congruence result for binomial coefficients.  Note that $$\Phi_{kc}(1) = \left\{\begin{array}{ll}p, &\mbox{ if }kc=p^j \mbox{ for some prime }p\\1, &\mbox{ if }kc \mbox{ is not a power of a prime}.\end{array}\right.$$  Of course $kc$ is a power of a prime only when $k=1$ and $c$ is a power of $2$, or when $c$ and $k$ are both powers of an odd prime $p$.  Taking $P(x)=1$ and $z=0$, our result therefore shows 

\begin{corollary}\label{cor:binomial.specialization}
If $p$ is an odd prime and $\alpha \geq 1$, then
$$\sum_{m \equiv j \imod{p^\alpha}} (-1)^{\frac{m-j}{p^\alpha}} \binom{n}{m} \equiv 0 \imod{p^f},$$ where $f = \left\lfloor\frac{n}{2p^\alpha}\right\rceil + \left\lfloor \frac{n}{2p^{\alpha+1}} \right\rceil+\cdots$.
Similarly,
$$\sum_{m \equiv j \imod{2^\alpha}} (-1)^{\frac{m-j}{2^\alpha}} \binom{n}{m} \equiv 0 \imod{2^{\left\lfloor \frac{n}{2^{\alpha+1}}\right\rceil}}.$$
\end{corollary}

When $p$ is an odd prime and $\alpha=1$, Corollary \ref{cor:binomial.specialization} has the same flavor as Fleck's original congruence; indeed, the sum under consideration will only differ from the sum in Fleck's congruence by at most a sign (since both are alternating sums of the same quantities).  Unfortunately, the previous result falls short of the number of factors of $p$ given in Fleck's congruence: $\left\lfloor \frac{n-1}{p-1}\right\rfloor$.  Indeed, in the long run Corollary \ref{cor:binomial.specialization} accounts for only half of the factors of $p$ in Fleck's original congruence.  This is shown in the following

\begin{proposition}
$$\lim_{n \to \infty} \frac{\left\lfloor \frac{n}{2p}\right\rceil + \left\lfloor \frac{n}{2p^2}\right\rceil + \cdots}{\left\lfloor \frac{n-1}{p-1}\right\rfloor} = \frac{1}{2}.$$
\end{proposition}

\begin{proof}
Let $L = \lfloor \log_p(n) \rfloor$.  Now $\lfloor n/2p^k \rceil = 0$ for $k>L$, so that 
$$\sum_{k=1}^{\infty} \left\lfloor \frac{n}{2p^k}\right\rceil  
= \sum_{k=1}^{L} \left(\frac{n}{2p^k}+O(1)\right) 
= \frac{n}{2p}\sum_{k=0}^{L-1}\frac{1}{p^k} + O(\log_p(n)) 
= \frac{n}{2(p-1)} + O(\log_p(n)).$$
Of course it is obvious that $\lfloor \frac{n-1}{p-1} \rfloor = \frac{n}{p-1} + O(1)$. The conclusion follows immediately.
\end{proof}

In \cite[Lem.~10]{We}, Weisman proved a variant of Fleck's congruence:
$$\sum_{m\equiv j \imod{p^\alpha}} (-1)^m \binom{n}{m} \equiv 0 \imod{p^{\left\lfloor \frac{n}{\phi(p^\alpha)}\right\rfloor-1}}.$$  This result is slightly weaker than Fleck's result when $\alpha=1$.  Weisman appears to have been unaware of Fleck's original work, but instead considered these sums in order to provide a counterexample to a question of Mahler concerning $p$-adic functions with continuous derivatives.  Though he gives credit to Weisman, it appears that Sun (\cite[Th.~1.1]{Su2}) was the first to prove the following generalization of Fleck's congruence: $$\sum_{m\equiv j \imod{p^\alpha}} (-1)^m \binom{n}{m} \equiv 0 \imod{p^{\left\lfloor \frac{n-p^{\alpha-1}}{\phi(p^\alpha)}\right\rfloor}}.$$ (In fact, Sun's result is far more general than what we state; his work also provides a generalization of another binomial congruence established by Wan in \cite{Wa} that arose from the study of $(\varphi,\Gamma)$-modules from Iwasawa theory.)  Corollary \ref{cor:binomial.specialization} provides divisibility results for Sun's sum when $p$ is odd, since in this case both sums are alternating.  Again, unfortunately, the corollary falls short of the known number of factors of $p$.  The proof that we gave in the previous proposition can be adapted to show that Corollary \ref{cor:binomial.specialization} again accounts for half of the factors of $p$ from Sun's theorem.

\subsection{Forced alternation when $p=2$}

In considering the specialization of Theorem \ref{th:fleck.at.k} to binomial coefficients, it's worth pointing out the sums considered in Fleck's congruence and Sun's extension are not alternating sums when $p =2$, $\alpha \in \Z_{>0}$: the sum is taken over those elements $m \equiv j \imod{2^\alpha}$ and the sign in the sum is contributed by $(-1)^m = (-1)^j$.  To capture alternating sums over congruence classes modulo $2^\alpha$, one needs to consider sums like the one found in Corollary \ref{cor:binomial.specialization}.  When one does so, one can use Sun's result as the key ingredient in the proof of the following result.  

\begin{proposition}\label{prop:forced.alternating.even.fleck}
For $\alpha$ a positive integer and $0 \leq j<2^\alpha$, we have $$\sum_{m \equiv j \imod{2^\alpha}} (-1)^{\frac{m-j}{2^\alpha}} \binom{n}{m} \equiv 0 \imod{2^{\left\lfloor \frac{n}{2^\alpha}\right\rfloor}}.$$
\end{proposition}

\begin{proof}
Begin by noticing that
$$\sum_{m \equiv j \imod{2^\alpha}} (-1)^{\frac{m-j}{2^\alpha}} \binom{n}{m} = \sum_{m \equiv j \imod{2^{\alpha+1}}} \binom{n}{m} - \sum_{m \equiv j \imod{2^{\alpha+1}}} \binom{n}{m+2^\alpha}.$$  Now Sun's result tells us that $$2\cdot \sum_{m \equiv j\imod{2^{\alpha+1}}} \binom{n}{m+2^\alpha} \equiv 0 \imod{2^{\left\lfloor\frac{n}{2^\alpha}\right\rfloor}},$$ and by adding this quantity to the previous equation we have the following congruence modulo $2^{\left\lfloor\frac{n}{2^\alpha}\right\rfloor}$:
\begin{equation}\label{eq:binomial.part1}\begin{split}\sum_{m \equiv j \imod{2^\alpha}} (-1)^{\frac{m-j}{2^\alpha}} \binom{n}{m} & \equiv \sum_{m \equiv j \imod{2^{\alpha+1}}} \binom{n}{m} + \sum_{m \equiv j \imod{2^{\alpha+1}}} \binom{n}{m+2^\alpha}.\end{split}\end{equation}  

Sun's result also tells us that \begin{equation}\label{eq:binomial.part2}\sum_{m \equiv j \imod{2^{\alpha+1}}} \binom{n+2^\alpha}{m+2^\alpha} \equiv 0 \imod{2^{\left\lfloor \frac{n}{2^\alpha}\right\rfloor}},\end{equation} though if we apply the Chu-Vandermonde identity with $t=2^\alpha$ we find
\begin{equation*}\begin{split}
\sum_{m \equiv j \imod{2^{\alpha+1}}} \binom{n+2^\alpha}{m+2^\alpha} = \sum_{k=0}^{2^\alpha} \binom{2^\alpha}{k}\sum_{m \equiv j \imod{2^{\alpha+1}}} \binom{n}{m+2^\alpha-k}.
\end{split}\end{equation*} Sun's result again tells us that
$$\sum_{m \equiv j \imod{2^{\alpha+1}}} \binom{n}{m+2^\alpha-k} \equiv 0 \imod{2^{\left\lfloor\frac{n-2^\alpha}{2^\alpha}\right\rfloor}},$$ and since $2 \mid \binom{2^\alpha}{k}$ for each $0 < k < 2^\alpha$ we have a congruence modulo $2^{\left\lfloor \frac{n}{2^\alpha}\right\rfloor}$:
\begin{equation}\label{eq:binomial.part3}\sum_{m \equiv j \imod{2^{\alpha+1}}} \binom{n+2^\alpha}{m+2^\alpha} \equiv \sum_{m \equiv j \imod{2^{\alpha+1}}} \binom{n}{m+2^\alpha} + \sum_{m \equiv j \imod{2^{\alpha+1}}} \binom{n}{m}.\end{equation}
Combining equations (\ref{eq:binomial.part1}),(\ref{eq:binomial.part2}) and (\ref{eq:binomial.part3}) completes the proof.
\end{proof}

There are a few things worth noting.  First, while there are only half as many factors of $2$ in Proposition \ref{prop:forced.alternating.even.fleck} compared to the non-alternating sums considered by Fleck or Sun's congruences, Corollary \ref{cor:binomial.specialization} still only accounts for half of the factors of $2$ from Proposition \ref{prop:forced.alternating.even.fleck}.  Second, the application of the Chu-Vandermonde in the previous proof yields one more factor of $2$ in the sum than those given by simply applying Sun's result to the first equation in the proof. 

Finally, and perhaps most interesting, whereas the $2$-divisibility from Sun's results was the key ingredient in the proof of our Proposition \ref{prop:forced.alternating.even.fleck}, one cannot expect a similar type of connection between alternating and non-alternating sums in the setting of $q$-binomial coefficients.  In particular, it seems that the non-alternating sum of $q$-binomial coefficients across a fixed congruence class modulo $2^\alpha$ will frequently have no factors of the form $\Phi_{2^\beta}$. For instance, one can check that
$$\sum_{m \equiv 1 \imod{4}} (-1)^m\binom{7}{m}_q = -\Phi_7(q) (q^4+q^2+2),$$ whereas
$$\sum_{m \equiv 1 \imod{4}} (-1)^{\frac{m-1}{4}} \binom{7}{m}_q = q^2 \Phi_4(q) \Phi_7(q).$$

\section{Future directions}\label{sec:wrapup}

One potential avenue for future exploration is investigating where the remaining factors of $p$ from Fleck and Sun's results are buried in the $q$-analog.  Certainly one could attempt to work through the various identities in this paper to give explicit formulae for the factors which remain after the ``predictable" cyclotomic factors are accounted for.  A similar question was asked by Sun and Wan in \cite{SW} in the case of binomial coefficients, and there are many satisfactory results in this setting. 

The following result gives a method for accounting for the multiplicity of various predictable cyclotomic factors for a general $n$, much in the same fashion as Proposition \ref{prop:fleck.restated}.  We omit the proof, since it is analogous to that of Proposition \ref{prop:fleck.restated}.

\begin{proposition}\label{prop:main.theorem.restated}
If $n, k \in \Z_{\geq 0}$ with $k$ odd, $c \in \Z_{>0}$, and $P \in \Z[\zeta_{2c}][x]$, then
$$\sum_{m=0}^n \zeta_{2c}^m P(m)(q^z)^m \frac{d^l}{dq^l}\left[\binom{n}{m}_q\right] \equiv 0 \imod{\mathop{\prod_{k \mbox{\tiny{ odd}}}}\Phi_{kc}(q)^{\varepsilon(kc,l,P(x),n)}},$$ where $\varepsilon(kc,l,P(x),n) = \left\lfloor \frac{n}{2kc}-\frac{\deg(P)}{2}-l\right\rceil$ (as defined prior to stating Proposition \ref{prop:fleck.restated}).
\end{proposition}

With the previous result in mind, let us define a family of polynomials, $$R = \{R^{(l)}_c(P(x),z,n) \} \subseteq \Z[\zeta_{2c}][q],$$ by
$$Q^{(l)}_c(P(x),z,n) = R^{(l)}_c(P(x),z,n) \mathop{\prod_{k \mbox{\tiny{ odd}}}}\Phi_{kc}(q)^{\varepsilon(kc,l,P(x),n)},$$ where $\varepsilon(kc,l,P(x),n) = \left\lfloor \frac{n}{2kc}-\frac{\deg(P)}{2}-l\right\rceil$.

In the following result we give a recursive relationship satisfied by the $R$ polynomials.  Though it is stated in a fairly special case, we point out that one could deduce more general recursive relationships either by using the full generality of Lemma \ref{le:n.in.terms.of.n.minus.one.and.two}, or by using Equation (\ref{eq:deriv.relation}), or possibly even by using the Chu-Vandermonde Analog (Lemma \ref{le:critical.identity}). 

\begin{lemma}\label{le:recursion}
When $l=z=0$ and $P(x)=1$, the polynomials $R^{(0)}_c(1,0,n)$ have the following recursive relationship:
$$R^{(0)}_c(1,0,n) \prod_{k \mbox{\tiny{ odd}}} \Phi_{kc}(q)^{\alpha(kc,n)}= (1+\zeta_{2c})R^{(0)}_c(1,0,n-1)\prod_{k \mbox{\tiny{ odd}}} \Phi_{kc}(q)^{\beta(kc, n)}- \zeta_{2c}(1-q^{n-1})R^{(0)}_c(1,0,n-2),$$ where $$\alpha(kc,n) = \left\{\begin{array}{ll}1,&\mbox{ if }$n=ikc$\mbox{ or }$n=ikc+1$ \mbox{ for some odd } $i$\\0, &\mbox{ otherwise}\end{array}\right.$$ and
$$\beta(kc,n) = \left\{\begin{array}{ll}1,&\mbox{ if }$n=ikc+1$ \mbox{ for some odd } $i$\\0, &\mbox{ otherwise}\end{array}\right.$$
\end{lemma}

\begin{proof}
This result follows by dividing Lemma \ref{le:n.in.terms.of.n.minus.one.and.two} by $\prod_{k \mbox{\tiny{ odd}}} \Phi_{kc}(q)^{\varepsilon(kc,0,1,n-2)}$ and observing that 
$$\left\lfloor \frac{n-1}{2kc} \right\rceil - \left\lfloor \frac{n-2}{2kc} \right\rceil = \beta(kc,n),$$ and that
$$\left\lfloor \frac{n}{2kc} \right\rceil - \left\lfloor \frac{n-2}{2kc} \right\rceil = \alpha(kc,n).$$
\end{proof}

\begin{remark*}
Notice that in the special case where $c=1$, the previous lemma gives 
$$R^{(0)}_1(1,0,n) \prod_{k \mbox{\tiny{ odd}}} \Phi_{k}(q)^{\alpha(k,n)} = (-1) R^{(0)}_1(1,0,n-2) \prod_{k \mid n-1} \Phi_{k}(q).$$ Of course one can easily compute that $R^{(0)}_1(1,0,1)=0$, and hence $R^{(0)}_1(1,0,n) = 0$ whenever $n$ is odd.  When $n$ is even notice that $\alpha(k,n) = 1$ only when $n=ik+1$ for some odd $i$.  But this means that $k \mid n-1$, and hence the additional appearance of $\Phi_k(q)$ on the left side of the equation is balanced by the contribution of a cyclotomic factor from $1-q^{n-1}$.  Since one can easily compute that $R^{(0)}_1(1,0,2)=-1 = (-1)^{2/2}$, we have $$R_1^{(0)}(1,0,n) = \left\{\begin{array}{ll}(-1)^{n/2},&\mbox{ if }n \mbox{ is even}\\0,&\mbox{ if }n \mbox{ is odd}.\end{array}\right.$$  Hence Lemma \ref{le:recursion} gives us another proof that the Gaussian Formula is a consequence of the proof of Theorem \ref{th:main.theorem}.
\end{remark*}

After running computations for $n,p \leq 50$, the authors have found that \begin{equation}\label{eq:q.fleck.sum}\sum_{m\equiv j \imod{p}}(-1)^\frac{m-j}{p}\binom{n}{m}_q\end{equation}  has at most one non-cyclotomic, irreducible factor other than powers of $q$ --- with the exception of $n=8$, $p=3$ and $j=1$.  It seems quite likely that the additional polynomial factors could be computed recursively in a fashion similar to our construction of the polynomials $R$ from Lemma \ref{le:recursion}.  Some preliminary calculations for these polynomials, however, do not provide any insight into why their factorizations include approximately $\frac{n}{2}$ factors of $p$ when $q=1$.  We include just a few computations to whet the reader's appetite in Table \ref{table:computations}, as produced by \emph{Mathematica}.

As a final remark, computations show that for every $p \leq n \leq 100$ there exists a $j$ so that for each $t \leq \lfloor \log_p(n)\rfloor$ the number of factors of $\Phi_{p^t}(q)$ in equation (\ref{eq:q.fleck.sum}) is, indeed, equal to the number of factors predicted in Theorem \ref{th:fleck.at.k}.   Hence it seems in general that the additional factors of $p$ from Fleck's original congruence arise from the aforementioned unpredicted, non-cyclotomic factor(s).  It also suggests that for $n\in \Z_{\geq 0}$, Theorem \ref{th:fleck.at.k} is as sharp as possible without putting further restrictions on $j$.  It would be quite interesting to characterize those $0 \leq j \leq p-1$ for which the prescribed cyclotomic multiplicities are not sharp, or to simply find values of $n$ and $j$ for which the prescribed multiplicity of $\Phi_{kc}(q)$ differs from the actual multiplicity of $\Phi_{kc}(q)$ by a wide margin.

\section{acknowledgements}
The authors wish to thank Bruce Berndt, J\'an Mi\v{n}\'{a}c, Bruce Reznick and Mike Zieve for their helpful comments in the preparation of this paper.

\begin{table}[b]
\caption{Some computations of non-cyclotomic factors in Fleck sums}\label{table:computations}
\begin{tabular}{p{2in}|p{4in}}
\begin{center}Fleck sum\end{center} & \begin{center}Non-cyclotomic factor(s)\end{center} \\
\hline
\\
$\displaystyle \sum_{m\equiv 1 \imod{3}}(-1)^{\frac{m-1}{3}} \binom{8}{m}_q$& $\displaystyle (q^3+q+1) (q^7+q^4+q^3+q-1)$\\[20pt]
$\displaystyle \sum_{m \equiv 1\imod{5}}(-1)^{\frac{m-1}{5}}\binom{21}{m}_q$&$\displaystyle q^{94} + q^{92} + 2q^{90} + 2q^{89} + 3q^{88} + 3q^{87} + 5q^{86} + 6q^{85} + 10q^{84} + 9q^{83} + 14q^{82} + 14q^{81} + 22q^{80} + 25q^{79} + 31q^{78} + 34q^{77} + 43q^{76} + 49q^{75} + 62q^{74} + 65q^{73} + 80q^{72} + 86q^{71} + 105q^{70} + 114q^{69} + 129q^{68} + 140q^{67} + 158q^{66} + 172q^{65} + 193q^{64} + 201q^{63} + 223q^{62} + 233q^{61} + 256q^{60} + 266q^{59} + 282q^{58} + 292q^{57} + 308q^{56} + 317q^{55} + 332q^{54} + 331q^{53} + 344q^{52} + 341q^{51} + 352q^{50} + 348q^{49} + 347q^{48} + 341q^{47} + 338q^{46} + 330q^{45} + 324q^{44} + 307q^{43} + 300q^{42} + 282q^{41} + 273q^{40} + 256q^{39} + 238q^{38} + 222q^{37} + 203q^{36} + 189q^{35} + 173q^{34} + 152q^{33} + 138q^{32} + 121q^{31} + 110q^{30} + 95q^{29} + 79q^{28} + 69q^{27} + 57q^{26} + 50q^{25} + 41q^{24} + 30q^{23} + 25q^{22} + 17q^{21} + 15q^{20} + 11q^{19} + 4q^{18} + 4q^{17} + 3q^{15} - q^{14} - 2q^{13} - 3q^{12} - 2q^{11} - 2q^{10} - q^{9} - 4q^{8} - q^{7} - 3q^{6} - q^{4} - 2q^{3} - q^{2} - q + 1$\\[20pt]
$\displaystyle \sum_{m\equiv 3\imod{7}}(-1)^{\frac{m-3}{7}}\binom{23}{m}_q$ & $\displaystyle q^{64} - q^{63} + q^{60} + q^{57} + q^{54} + 2q^{52} + 2q^{50} + 2q^{48} + q^{47} + 2q^{46} + q^{45} + 3q^{44} + 2q^{43} + 2q^{42} + q^{41} + 4q^{40} + 2q^{39} + 3q^{38} + 2q^{37} + 4q^{36} + 2q^{35} + 4q^{34} + 2q^{33} + 3q^{32} + 2q^{31} + 4q^{30} + 3q^{29} + 3q^{28} + 2q^{27} + 2q^{26} + 2q^{25} + 3q^{24} + 2q^{23} + 3q^{22} + 2q^{20} + q^{19} + 2q^{18} + q^{17} + q^{16} + q^{15} + q^{13} + q^{12} + q^{8} + q - 1$
\end{tabular}
\end{table}


\begin{thebibliography}{XXX}

\bibitem{A}\label{A} G.~Andrews. $q$-analogs of binomial coefficient congruences of Babbage, Wolstenholme and Glaisher. \textit{Discrete Math.} {\bf 204}(1999), no.~1, 15--25. 

\bibitem{AE}\label{AE} G.~Andrews, K.~Eriksson. \textit{Integer Partitions}. Cambridge University Press, 2004.

\bibitem{C}\label{C} W.E.~Clark. $q$-analogue of a binomial coefficient congruence. \textit{Int.~J.~Math.~and Math.~Sci.}, {\bf 18}(1995), no.~1, 197--200.  

\bibitem{D}\label{D} L.E.~Dickson. \textit{History of the Theory of Numbers}, Vol.~I. New York, NY: Chelsea Publishing, 1966.  

\bibitem{F}\label{F} A.~Fleck. \textit{Sitzungs. Berlin Math.~Gesell.}, {\bf 13} (1913-4), 2--6. 

\bibitem{G}\label{G} A.~Granville. Arithmetic properties of binomial coefficients I: binomial coefficients modulo prime powers. \textit{Canadian Math. Society Conf. Proc.}, {\bf 20}(1997), 253--275.

\bibitem{HS}\label{HS} H.~Hu, Z.W.~Sun. An extension of Lucas' theorem. \textit{Proc.~Amer.~Math.~Soc.} {\bf 129}(2001), no.~12, 3471--3478.

\bibitem{LW}\label{LW} J.H.~van Lint, R.M~Wilson. \textit{A Course in Combinatorics.} 2nd. ed., Cambridge University Press, 2001.

\bibitem{P}\label{P1} H.~Pan. A $q$-analogue of Lehmer's congruence. \textit{Acta Arith.}, {\bf 128}(2007), no.~4, 303--318.

\bibitem{PFleck}\label{P2} H.~Pan. Factors of Some Lacunary q-Binomial Sums.  Available on the ArXiv at \url{https://arxiv.org/abs/1201.5963}.

\bibitem{Sa}\label{Sa} B.~Sagan. Congruence properties of $q$-analogs. \textit{Adv.~Math.}, {\bf 95}(1992), no.~1, 127--143.


\bibitem{Su1}\label{S1} Z.W.~Sun. On the sum $\sum_{k \equiv r \imod{m}}\binom{n}{k}$ and related congruences. \textit{Israel J.~Math.} {\bf 128}(2002), 135--156.

\bibitem{Su2}\label{S2} Z.W.~Sun. Polynomial extension of Fleck's congruence. \textit{Acta Arith.}~{\bf 122}(2006), no.~1, 91--100.

\bibitem{SD}\label{SD} Z.W.~Sun, D.~Davis. Combinatorial congruences modulo prime powers. \textit{Trans.~Amer.~Math.~Soc.} {\bf 359}(2007), no.~11, 5525--5553.

\bibitem{SW}\label{SW} Z.W.~Sun, D.~Wan. On Fleck quotients. \textit{Acta Arith.} {\bf 127}(2007), no.4, 337--363. 

\bibitem{Wa}\label{Wa} D.~Wan. Combinatorial congruences and $\psi$-operators. \textit{Finite Fields \& App.}. {\bf 12}(2006), 693-703.

\bibitem{We}\label{We} C.S.~Weisman.  On $p$-adic differentiability. \textit{J.~Number~Theory} {\bf 9}(1977), no.~1, 79--86.

\end{thebibliography}
\end{document}